\documentclass[a4paper, twoside, 10pt]{article}
\usepackage{eurosym}
\usepackage{amsmath,amsthm,amsfonts,latexsym,amscd,amssymb, enumerate}
\usepackage{hyperref, xcolor}
\numberwithin{equation}{section}
\input xypic
\newtheorem{theorem}{Theorem}[section]
\newtheorem{lemma}{Lemma}[section]
\newtheorem{corollary}{Corollary}[section]
\newtheorem{proposition}{Proposition}[section]
\newtheorem{definition}{Definition}[section]

\theoremstyle{remark}
\newtheorem{remark}{Remark}[section]

\newcommand{\cl}{\mathcal{L}}
\newcommand{\cs}{\mathcal{S}}

\newcommand{\cri}{\mathcal{R}}
\newcommand{\clb}{\overline{\mathcal{L}}}
\newcommand{\csb}{\overline{\mathcal{S}}}
\newcommand{\cbb}{\overline{\mathcal{B}}}
\newcommand{\M}{\overline{M}}
\newcommand{\ch}{\mathfrak{h}}

\date{}
\title{\textbf{Submanifolds of Bochner Holomorphic Statistical Manifolds}}
\author{Harmandeep Kaur$^1$\thanks{harmandeepkaur1559@gmail.com} ~and Gauree Shanker$^1$\thanks{corresponding author, Email: gauree.shanker@cup.edu.in}
\vspace{.3cm}\\
$^1$ Department of Mathematics and Statistics,\\
 Central University of Punjab, Bathinda, Punjab-151401, India.}

\begin{document}
\maketitle
\begin{center} \textbf{Abstract} \end{center}
	In this paper, we study the submanifolds of holomorphic statistical manifolds with vanishing Bochner curvature tensor (Bochner holomorphic statistical manifolds). We study the Lagrangian statistical submanifolds and discuss their conformal flatness under the assumption of vanishing Bochner curvature tensor. Next we study the holomorphic submanifolds of Bochner holomorphic statistical manifolds and prove that the doubly autoparallel holomorphic submanifold of a Bochner holomorphic statistical manifold has vanishing Bochner curvature tensor. \\ %

   \section{Introduction}\label{hssec1}
Information geometry has provided a powerful geometric view point to understand statistical models. In this context, the theory of statistical manifolds emerges as a natural bridge between differential geometry, probability theory, affine geometry and Hessian geometry. The geometric formulation of statistical inference was initiated by Rao through the Fisher information metric \cite{bc12}. Lauritzen introduced the abstract notion of a statistical manifold \cite{bc10}, while Amari developed the modern theory and established its geometrical framework \cite{bc13,bc14}. From statistical point of view, statistical manifolds are the Riemannian manifolds whose points represent probability distributions. Geometrically, a statistical manifold is a smooth manifold $M$ endowed with Riemannian metric $g$ and torsion free affine connections $(\nabla, \nabla^*)$ defined by $$Xg(Y, Z) = g( \nabla_XY, Z) +  g(Y, \nabla^*_XZ)$$ for arbitrary vector fields $X, Y, Z$ on $M$ such that the $(0, 3)$-tensor field $\nabla g$ is totally symmetric. When the dual connection $\nabla^*$ conincides with $\nabla$, the affine connection becomes the Levi-Civita connection and the statistical manifold reduces to a Riemannian manifold. Inspired by the different types of complex and contact structures on Riemannian manifolds \cite{bc30}, the theory of statistical manifolds was further developed, giving rise to complex statistical manifolds and contact statistical manifolds. In this line of research, Kurose first introduced the complex version of statistical manifolds, called holomorphic statistical manifolds, which was further studied systematically by Furuhata \cite{bc1, bc2}. A holomorphic statistical manifold $\M$ is an almost complex manifold with almost complex structure $J$ and a statistical structure $(\overline{\nabla}, g)$ such that $\omega$ is a $\overline{\nabla}$-parallel $2$-form on $\M$ defined by $\omega(X,Y)= g(X,JY)$. The theory of holomorphic statistical manifolds lie at the intersection of complex geometry, information geometry, Hessian geometry and Riemannian geometry making it a interesting field of research.\\
On the other hand, the study of conformal nature of smooth structures has been a fundamental aspect of geometry. In this regard, Weyl curvature tensor, named after Hermann Weyl, is a fundamental tensor for studying the conformal flatness in dimension $\geq4.$ It plays an important role in general theory of relativity and pseudo-Riemannian geometry as it is a measure of the curvature of spacetime. In particular, it is the invariant part of the Riemannian curvature tensor under a conformal transformation of the metric. Building on Weyl's work, Bochner in \cite{bc5} introduced the complex couterpart of the Weyl curvature tensor known as the  Bochner curvature tensor. The study of Bochner curvature tensor  is particularly interesting in complex geometry, since it plays a role analogous to the Weyl curvature tensor in Riemannian geometry, and its vanishing corresponds to the complex analogue of conformal flatness. Tachibana \cite{bc6}, further, studied the Bochner curvature tensor by giving its components with respect to real local coordinates. Following its significance, the study of K\"{a}hler manifolds with vanishing Bochner curvature (Bochner K\"{a}hler manifolds) and their submanifolds emerged as an important area of research in complex geometry \cite{bc111, bc16, bc17, bc18}. However, it has not yet been explored in the setting of statistical manifolds. Since holomorphic statistical manifolds are generalization of K\"{a}hler manfolds, it is natural to study their conformal flatness using the statistical Bochner curvature tensor.\\ 
To address this gap in the literature, the aim of this paper is to study the geometry of submanifolds of holomorphic statistical manifolds with vanishing Bochner curvature tensor, which we call Bochner holomorphic statistical manifolds. In particular, we study Lagrangian statistical submanifolds and holomorphic submanifolds of Bochner holomorphic statistical manifolds. Our aim is to prove the following:

  \begin{theorem}\label{hstheorem1} 
       Let $\M$ be a holomorphic statistical manifold of real dimension ${2m}, m \geq 4$ with vanishing Bochner curvature tensor. Then a doubly autoparallel Lagrangian statistical submanifold $M$ of $\M$ is conformally flat.
        \end{theorem}

  \begin{theorem}\label{hstheorem2} 
      Let $M$ be a $2n$ dimensional doubly autoparallel holomorphic submanifold of a holomorphic statistical manifold $\M$ of dimension $2m$. If $\M$ has vanishing Bochner curvature tensor, then the Bochner curvature tensor of $M$ also vanishes.
  \end{theorem}

\begin{remark}
For convenience, we denote holomorphic statistical manifold by HSM throughout the paper.
\end{remark}

       \section{Preliminaries}\label{hssec2}
  In this section, we recall the preliminary notions and some important results required for our study. \\
Let $(M, \nabla, g)$ be a statistical manifold. Then, the curvature tensor field for affine connection $\nabla$ is defined as
\begin{align}
\cri(X, Y)Z = \nabla_X\nabla_YZ - \nabla_Y\nabla_XZ - \nabla_{[X, Y]}Z
\end{align}
for $X, Y, Z \in \Gamma(TM).$ We denote by $\cri^*$ and $\cri^g$, the curvature tensor fields for $\nabla^*$ and $\nabla^g$, respectively. Also the statistical curvature tensor $\cs \in \Gamma(TM)^{(1, 3)}$ is defined as \cite{bc2}
 \begin{align}\label{hssc1}
\cs(X, Y)Z = \dfrac{1}{2}[\cri(X, Y)Z + \cri^*(X, Y)Z].
 \end{align}

  \begin{definition}\label{hsssc}\cite{bc2}
For any $x \in M$ and a two-dimensional subspace $\Omega = span_{\mathbb{R}}\{u,v\}$ of $T_xM$, the statistical sectional curvature $\mathcal{K}$ is given by
  \begin{align*}
  \mathcal{K}(u \wedge v) = \dfrac{g(\cs(u, v)v, u)}{g(u, u)g(v, v) - (g(u, v))^2}. 
   \end{align*}
   $M$ is said to be of constant statistical sectional curvature $k (\in \mathbb{R})$ if $\mathcal{K} = k ~\forall ~x \in M$ and $\forall~ \Omega$.
      \end{definition}

\begin{remark}\cite{bc2}
The sectional curvature of a statistical manifold $(M, \nabla, g)$ is constant, say $k$ if and only if for $X, Y, Z \in \Gamma(TM),$ the statistical curvature tensor field is given by
  \begin{align}\label{hssecc1}
\cs(X, Y)Z = k\{g(Y, Z)X - g(X, Z)Y\}.
  \end{align}
\end{remark}
 \begin{definition}\label{hsrc}\cite{bc2}
Let $(M, \nabla, g)$ be an $n$-dimensional statistical manifold. Choose a local orthonormal frame $\{e_1, ..., e_n\}$ on $M.$ Then for $X, Y, Z \in \Gamma(TM),$ the statistical Ricci curvature $\cl \in \Gamma(TM^{(0, 2)}),$ statistical scalar curvature $\rho \in C^\infty(M)$ are defined, respectively, as        
          \begin{align}%\label{hsricc}
   &\cl(Y, Z) = tr\{X \mapsto \cs(X, Y)Z\} = \sum_{i=1}^{n}g(\cs(e_i, Y)Z, e_i), \label{hsricc} \\
   &\rho = tr_g\cl = \sum_{i=1}^{n}\cl(e_i, e_i). \label{hsscal}
      \end{align}
      \end{definition}
For a HSM $(\M,\overline{\nabla}, g, J)$, the following relations hold for $X, Y, W \in \Gamma(T\M),$
  \begin{align} \label{hseq11}
\overline{\nabla}_X(JY) = J\overline{\nabla}^*_XY,~\cri(X, Y)JZ = J\cri^*(X, Y)Z,~ \csb(X, Y)JZ = J\csb(X, Y)Z. 
   \end{align}
  \begin{lemma}\cite{bc2}\label{statlemma1}
     Let $(\M, \overline{\nabla}, g, J)$ be a HSM. Then, for $X, Y, W, Z \in \Gamma(TM),$ the following holds:
     \begin{equation}
     g(\csb(X, Y)W, Z) = g(\csb(X, Y)JW, JZ) = g(\csb(JX, JY)W, Z).
     \end{equation}
     \end{lemma} 
  \begin{definition}\label{hswcc1} 
The Weyl curvature tensor $\mathcal{C}^s$ for statistical manifolds is given by
\begin{align}\label{hsw2}
\mathcal{C}^s(X, Y)Z =& ~\csb(X, Y)Z - \dfrac{1}{(n-2)} [g(Y, Z)\clb X - g(\clb X, Z)Y + g(\clb Y, Z)X \nonumber \\
& - g(X, Z)\clb Y] + \frac{\overline\rho}{(n-2)(n-1)}[g(Y, Z)X - g(X, Z)Y].
\end{align}
\end{definition}

      Next we define the statistical Bochner curvatue tensor in a way similar to the Bochner curvatue tensor defined by Bochner in \cite{bc5}.
      \begin{definition}\label{hsboc1}
      Let $\M$ be a HSM of real dimension $2m$. Let $\csb, \clb, \overline{\rho}$ denote the statistical curvature tensor, statistical Ricci curvature tensor and statistical scalar curvature tensor of $\M$, respectively. Then the statistical Bochner curvatue tensor $\cbb$ of type $(1, 3)$ on vector fields $X, Y, Z \in \Gamma(\overline{TM})$ is given by
      \begin{align}\label{hsbf1}
  \cbb(X, Y)Z =& \csb(X, Y)Z - \frac{1}{2m+4}\{g(Y, Z)\clb X - g(\clb X, Z)Y + g(\clb Y, Z)X \nonumber \\ 
  &- g(X, Z)\clb Y + g(JY, Z)\clb JX - g(\clb JX, Z)JY + g(\clb JY, Z)JX \nonumber \\ 
  &-  g(JX, Z)\clb JY - 2 g(JX, \clb Y)JZ - 2 g(JX, Y)\clb JZ \} \nonumber \\ 
  &+ \frac{\overline{\rho}}{(2m+4)(2m+2)}\{g(Y, Z)X - g(X, Z)Y + g(JY, Z)JX \nonumber \\ 
  &- g(JX, Z)JY - 2g(JX, Y)JZ\}.
      \end{align}
      \end{definition}

    \begin{definition}\cite{bc4}\label{hsEinstein}    
  A statistical manifold $\M$ is said to be statistical Einstein manifold if the statistical Ricci curvature $\clb(X,X)$ is constant for any unit vector field $X \in \Gamma(T\M).$
    \end{definition}           
     
    \begin{proposition}\cite{bc4}\label{hsEinstein12} 
      Let $(\M, \overline{\nabla}, g)$ be a statistical Einstein manifold of dimension $n$. Then the statistical Ricci curvature tensor $\clb$ of $\M$ is of the form $\clb = \dfrac{\overline\rho}{n} g,$ where $\overline{\rho}$ denote the statistical scalar curvature of $\M$.
       \end{proposition}
     
     \begin{definition}\cite{bc2}\label{hshsec}
     A HSM $(\M, \nabla, g, J)$ is said to be of constant holomorphic sectional curvature $c (\in\mathbb{R})$ if the sectional curvature is constant $c$ for any $x \in \M$ and for any $J$-invariant two dimensional subspace of $T_x\M.$
       \end{definition}       
    
    \begin{remark}\cite{bc2}\label{hssec11}
    If a HSM $(\M, \nabla, g, J)$ has constant holomorphic sectional curvature then for $X, Y, Z \in \Gamma(T\M),$ the statistical curvature tensor is given by
    \begin{align}\label{hsre1}
   \csb(X, Y)Z = \dfrac{c}{4} [g(Y, Z)X - g(X, Z)Y + g(JY, Z)JX - g(JX, Z)JY - 2 g(JX, Y)JZ].   
       \end{align}
    \end{remark}
      
Let $M$ be an $n$-dimensional isometrically immersed submanifold of a HSM $(\M, \nabla, g, J)$. Let us denote the induced metric on $M$ by the same symbol $g$. Then the tangent space $T_x\M$ at $x \in \overline{M}$ has the following orthogonal decomposition in terms of tangent and normal space of $M$
      \begin{equation*}
      T_x\overline{M} = T_xM \oplus (T_xM)^{\perp}.
      \end{equation*}
   Accordingly, for $X, Y \in \Gamma(TM)$ and $\xi \in \Gamma(TM)^{\perp}$, the Gauss and Weingarten formulae are, respectively, given by \cite{bc7}
     \begin{align}
   & \overline{\nabla}_XY = {\nabla}_XY + h(X, Y), ~~ \overline{\nabla}^*_XY = {\nabla}^*_{X}Y + h^*(X, Y), \label{hsgf1} \\
   &  \overline{\nabla}_X\xi = -\mathcal{A}^*_{\xi}X + \mathcal{D}_X{\xi}, ~~~~~~ \overline{\nabla}^*_X\xi = -\mathcal{A}_{\xi}X + \mathcal{D}^*_X{\xi}, \label{hsgf2} 
     \end{align}
where $h$ and $h^*$ are the imbedding curvature tensors of $M$ in $\M$ for $\overline{\nabla}$ and $\overline{\nabla}^*$, respectively. ${\mathcal{A}}^*_{\xi}X$, $\mathcal{A}_{\xi}X$ are tangential components and $\mathcal{D}_X{\xi},$ $ \mathcal{D}^*_X{\xi}$ are the normal components of $\overline{\nabla}_X\xi ,$ $ \overline{\nabla}^*_X\xi$, respectively. Also, we have
    \begin{align}
      g({\mathcal{A}}_{\xi}X, Y) = g(h(X, Y), \xi) , ~~g({\mathcal{A}}^*_{\xi}X, Y) = g(h^*(X, Y), \xi). \label{hswf1}
 %   &mm \label{hswf2}.
    \end{align}

    The Gauss equation is given by \cite{bc4}
    \begin{align}\label{hsge1}
   2 g(\csb(X, Y)Z, W) =&~ 2 g(\cs(X, Y)Z, W) - g(h(Y, Z), h^{*}(X, W)) \nonumber \\ 
   &+ g(h(X, Z), h^{*}(Y, W)) - g(h^{*}(Y, Z), h(X, W)) \nonumber \\ 
   &+ g(h^{*}(X, Z), h(Y, W)).
    \end{align}  	
    %%%%%%%%%%%%%%%%%%%%%%%%%%%%%%%%%%%%%%%%%%%%%%%%%%%%%%%%%%%%%%%%%%
       \section{Main Results}\label{hssec3}
   In this section we prove the main results. We start with an important lemma that is required throughout our study.    
     \begin{lemma}\label{hslemma1}
      The statistical Ricci curvature tensor $\clb$ of a HSM is complex linear i.e., $J\clb = \clb J$.
      \end{lemma}
      \begin{proof}
      Let $\M$ be a HSM of real dimension $2m.$ Choose a local orthonormal frame $\{e_1, ..., e_m, Je_1, ..., Je_m\}$ of $\overline{M}.$ Then, for $U, V \in \Gamma(T\M),$ the statistical Ricci curvature tensor $\clb$ is given by
      \begin{align}\label{hsle1}
      \clb(U, V) =& \sum_{i=1}^{m}g(\csb(e_i, U)V, e_i) + \sum_{i=1}^{m}g(\csb(Je_i, U)V, Je_i)
      \end{align}
    Using lemma \ref{statlemma1} and \eqref{hseq11} in \eqref{hsle1}, we get
        \begin{align*}%\label{hsle211}
      \clb(U, V) &= \sum_{i=1}^{m}g(\csb(e_i, U)JV, Je_i) - \sum_{i=1}^{m}g(\csb(Je_i, U)JV, e_i) \nonumber \\
      &= \sum_{i=1}^{m}g(\csb(U, e_i)Je_i, JV) + \sum_{i=1}^{m}g(\csb(Je_i, U)e_i, JV) \nonumber \\
       &= -\sum_{i=1}^{m}g(\csb(e_i, Je_i)U, JV).
       \end{align*}  
%Using (symmetries of S), we obtain
%        \begin{align*}%\label{hsle2}
%      \clb(U, V) =& \sum_{i=1}^{m}g(\csb(U, e_i)Je_i, JV) + \sum_{i=1}^{m}g(\csb(Je_i, U)e_i, JV),
%       \end{align*}  
%       which implies (using second Bianchi identity)
%        \begin{align*}%\label{hsle3}
%      \clb(U, V) &= -\sum_{i=1}^{m}g(\csb(e_i, Je_i)U, JV). %  \nonumber \\
%          \end{align*}  
            Using \eqref{hseq11} in above equation, we obtain
             \begin{align}\label{hsle4} 
       \clb(U, V) &= \sum_{i=1}^{m}g(\csb(e_i, Je_i)JU, V).   %\nonumber \\
  \end{align}  
  Replacing $U$ with $JU$ and $V$ with $JV$ in above equation, we have
     \begin{align}\label{hsle5} 
       \clb(JU, JV) &= \sum_{i=1}^{m}g(\csb(e_i, Je_i)J^2U, JV) = \sum_{i=1}^{m}g(\csb(e_i, Je_i)JU, V).  
  \end{align}  
  From \eqref{hsle4} and \eqref{hsle5}, we get
    \begin{align}\label{hsle6} 
     \clb(JU, JV) = \clb(U, V).
      \end{align}  
      Now we have
        \begin{align}\label{hsle7}
 g(\clb(JU), V) = \clb(JU, V) = \clb(J^2U, JV) = -\clb(U, JV).
   \end{align}
   Also we have
    \begin{align}\label{hsle8}
 g(J(\clb U), V) = -g(\clb(U), JV) = -\clb(U, JV)
   \end{align}
   From \eqref{hsle7} and \eqref{hsle8}, we get
    \begin{align}\label{hsle9}
 g(J(\clb U), V) =  g(\clb(JU), V), 
    \end{align}
     which proves the assertion.
      \end{proof}
      
   \begin{proposition}\label{hslemma3}
       Let $\M$ be a HSM of real dimension $2m$  with vanishing Bochner curvature tensor. If $\M$ is an Einstein statistical manifold, then $\M$ has constant holomorphic sectional curvature.
              \end{proposition}
      \begin{proof}
    Since $\M$ is an Einstein statistical manifold and has vanishing Bochner curvature tensor, therefore for $X, Y, Z \in \Gamma(T\M),$ using Proposition \ref{hsEinstein12} and equation \eqref{hsbf1}, we obtain 
      \begin{align*}%\label{hsl3.21}
       \csb =& \dfrac{\overline{\rho}}{2m(2m+4)} [2g(Y, Z)X - 2g(X, Z)Y + 2g(JY, Z)JX - 2g(JX, Z)JY \nonumber \\
       &- 4g(JX, Y)JZ] - \dfrac{\overline{\rho}}{(2m+4)(2m+2)} [g(Y, Z)X - g(X, Z)Y + g(JY, Z)JX.  \nonumber \\
        &- g(JX, Z)JY - 2 g(JX, Y)JZ],
         \end{align*}
     which on further simplification gives  
      \begin{align}\label{hsl3.111}
   \csb = \dfrac{\overline{\rho}}{4(m+1)} [g(Y, Z)X - g(X, Z)Y + g(JY, Z)JX - g(JX, Z)JY - 2 g(JX, Y)JZ].   
       \end{align} 
      Using Remark \ref{hssec11} and \eqref{hsl3.111}, we get the assertion.
     \end{proof}
   Next we prove our main result for Lagrangian statistical submanifolds of Bochner holomorphic statistical manifolds.           

\begin{proof}[Proof of Theorem \ref{hstheorem1}]
Consider a local orthonormal frame field $\{e_1, ..., e_m\}$ of vector fields tangent to $M$ and $\{Je_1, ..., Je_m\}$, a local orthonormal frame field of normal vector fields to $M$ in $\M,$ such that $\{e_1, ..., e_m, e_{m+1} = Je_1, ..., e_{2m} = Je_m\} $is a local orthonormal frame field of $T\M.$ Since Bochner curvature vanishes, for $Y, W \in \Gamma(T \M)$ using \eqref{hsboc1}, we obatin
 \begin{align}\label{hsth1eq1}
      \sum_{i=1}^{m}g(\csb(Je_i, Y)Je_i, W) =& \frac{1}{2(m+2)} \sum_{i=1}^{m}\{-g(\clb Je_i, Je_i)g(Y, W) - g(\clb Y, W) \nonumber \\
      &- g(Y, e_i)g(\clb e_i, W) - g(\clb JY, Je_i)g(e_i, W)  \nonumber \\
      &- 2g(e_i, \clb Y)g(e_i, W) - 2g(e_i, Y)g(\clb e_i, W)\}  \nonumber \\
      &+ \frac{\overline\rho}{4(m+1)(m+2)}\sum_{i=1}^{m}\{g(Y, W) + g(JY, Je_i)g(e_i, W)  \nonumber \\
      &+ 2g(Y, e_i)g(e_i, W)\}.                                                                                                                                                                                                 
         \end{align} 
Using Lemma \ref{hslemma1} and (compatibility of J) in \eqref{hsth1eq1}, we get
 \begin{align}\label{hsth1eq2}
  \sum_{i=1}^{m}g(\csb(Je_i, Y)Je_i, W) =& \frac{1}{2(m+2)}\sum_{i=1}^{m}\{-g(\clb e_i, e_i)g(Y, W) - g(\clb Y, W) \nonumber \\
      &- g(Y, e_i)g(\clb e_i, W) - g(\clb Y, e_i)g(e_i, W)  \nonumber \\
      &- 2g(e_i, \clb Y)g(e_i, W) - 2g(e_i, Y)g(\clb e_i, W)\}  \nonumber \\
      &+ \frac{(m+3) ~\overline\rho}{4(m+1)(m+2)} g(Y, W),
   \end{align} 
      which implies
     \begin{align}\label{hsth1eq3}
  \sum_{i=1}^{m}g(\csb(Je_i, Y)Je_i, W) =& \frac{\overline\rho}{(m+1)(2m+4)} g(Y, W) - \frac{m+6}{2m+4}g(\clb Y, W).
    \end{align} 
Also from \eqref{hsricc}, \eqref{hsge1} and \eqref{hsth1eq3}, we obtain the statistical Ricci curvature as
  \begin{align}\label{hsth1eq4}
  g(\clb Y, W) =& \frac{2(m+2)}{(m-2)}g(\cl Y, W) - \frac{\overline\rho}{(m-2)(m+1)}g(Y, W) %\nonumber \\
  + \frac{(m+2)}{(m-2)} \ch(Y, W),
    \end{align} 
    where we have used shorthand notation $\ch(Y, W)$ defined as
    \begin{align}\label{hsth1sn1}
    \ch(Y, W) =& \sum_{i=1}^{m}\{g(h(e_i, Y), h^*(e_i, W)) - g(h(e_i, e_i), h^*(Y, W)) \nonumber \\
  &+ g(h^*(e_i, Y), h(e_i, W)) - g(h^*(e_i, e_i), h(Y, W))\}.
     \end{align} 
    Using \eqref{hsscal}, \eqref{hsth1eq4} and \eqref{hsth1sn1}, we obtain the statistical scalar curvature $\overline\rho$ of $\M$ as
     \begin{align}\label{hsth1eq5}
  \overline\rho =& \dfrac{4(m+1)}{(m-1)}\rho + \frac{2(m+1)}{(m-1)} \sum_{j=1}^{m} \ch(e_j, e_j).
%  \{2 g(h(e_i, e_j), h^*(e_i, e_j)) - g(h(e_i, e_i), h^*(e_j, e_j)) \nonumber \\ &- g(h^*(e_i, e_i), h(e_j, e_j))\}. 
     \end{align}
     From \eqref{hsth1eq4} and \eqref{hsth1eq5}, we have
       \begin{align}\label{hsth1eq6}
    g(\clb Y, W) =& \frac{2(m+2)}{(m-2)}g(\cl Y, W) - \frac{4~\rho}{(m-1)(m-2)}g(Y, W) + \frac{(m+2)}{(m-2)} \ch(Y, W)  \nonumber \\
     &- \frac{2}{(m-1)(m-2)}\sum_{j=1}^{m} \ch(e_j, e_j)g(Y, W).
   \end{align}   
% \nonumber \\ &+ \frac{(m+2)}{(m-2)} \sum_{i=1}^{m}\{g(h(e_i, Y), h^*(e_i, W)) - g(h(e_i, e_i), h^*(Y, W)) \nonumber \\
%  &+ g(h^*(e_i, Y), h(e_i, W)) - g(h^*(e_i, e_i), h(Y, W))\} \nonumber \\
%   &- \frac{2}{(m-1)(m-2)}\sum_{i, j=1}^{m}\{2 g(h(e_i, e_j), h^*(e_i, e_j)) - g(h(e_i, e_i), h^*(e_j, e_j)) \nonumber \\
%  &- g(h^*(e_i, e_i), h(e_j, e_j))\} g(Y, W).    
 Since $\M$ has vanishing Bochner curvature tensor. Hence, for $X, Y, Z, W \in \Gamma(TM),$ using \eqref{hsboc1}, \eqref{hsge1} and \eqref{hsth1eq6}, we obtain
  \begin{align}\label{hsth1eq7}    
  0 =&~ 2g(\cs(X, Y)Z, W) - g(h(Y, Z), h^{*}(X, W)) + g(h(X, Z), h^{*}(Y, W)) \nonumber \\ 
   &- g(h^{*}(Y, Z), h(X, W)) + g(h^{*}(X, Z), h(Y, W)) \nonumber \\ 
   &- \frac{2}{m-2} [g(Y, Z)g(\cl X, W) - g(Y, W)g(\cl X, Z) + g(X, W)g(\cl Y, Z)  \nonumber \\ 
   &- g(X, Z)g(\cl Y, W)] + \frac{2~\rho}{(m-1)(m-2)}[g(Y, Z)g(X, W) - g(X, Z)g(Y, W)] \nonumber \\ 
   &\frac{1}{(m-1)(m+2)}[g(Y, Z)g(X, W) - g(X, Z)g(Y, W)]\sum_{j=1}^{m}\ch(e_j, e_j) \nonumber \\ 
   & - \frac{1}{m+2}[g(Y, Z) \alpha(X, W) - g(Y, W) \alpha(X, Z) + g(X, W) \alpha(Y, Z) \nonumber \\  
   &- g(X, Z) \alpha(Y, W)],
    \end{align}     
    where $\alpha(U, V)$ stands for 
      \begin{align}\label{hsth1sn2} 
  \alpha(U, V) = \frac{(m+2)}{(m-2)} \ch(U, V) - \frac{2}{(m-1)(m-2)}\sum_{j=1}^{m} \ch(e_j, e_j)g(U, V).
      \end{align}  
      As $M$ being doubly autoparallel implies the vanishing of $h$ and $h^*$, therefore from \eqref{hsw2} and \eqref{hsth1eq7}, it follows that the Weyl conformal curvature tensor $\mathcal{C}^s$ of $M$ vanishes. Hence, the assertion follows. Also we observe that if $M$ is not doubly autoparallel, but the terms containing $h$ and $h^*$ in the expression \eqref{hsth1eq7} vanish, then $M$ is conformally flat.
      %the assertion follows. 
\end{proof}
%%%%%%%%%%%%%%%%%%%%%%%%%%%%%%%%%%%%%%%%%%%%%%%%%%%%%%%%%%%%%%%%%%
\begin{corollary}
 Let $\M$ be a HSM of real dimension ${2m}, m \geq 4$ with vanishing Bochner curvature tensor. If $\M$ has constant holomorphic sectional curvature $\kappa$ then a doubly autoparallel Lagrangian statistical submanifold $M$ of $\M$ is a space of constant curvature $\frac{\kappa}{4}$.
\end{corollary}
\begin{proof}
Choose a local orthonormal frame $\{e_1, ..., e_m, Je_1, ..., Je_m\}$ of $\overline{M}.$ Since $\M$ has constant holomorphic sectional curvature, using \eqref{hsre1} for $Y, Z \in \Gamma(T\M),$ the statistical Ricci tensor $\clb$ is given by 
 \begin{align}\label{hscor1}
g(\clb Y, Z) &=\sum_{i=1}^{m}g(\csb(e_i, Y)Z, e_i) + \sum_{i=1}^{m}g(\csb(Je_i, Y)Z, Je_i)  \nonumber \\  
&= \sum_{i=1}^{m}\dfrac{\kappa}{4} [2 g(Y, Z) + 2g(Y, e_i)g(Z, e_i) + 2g(Y, Je_i)g(Z, Je_i)]   \nonumber \\
&= \sum_{i=1}^{m}\dfrac{\kappa (m+1)}{2} g(Y, Z).
  \end{align}     
Also, the statistical scalar curvature $\overline\rho$ of $\M$ is given as
 \begin{align}\label{hscor2}
 \overline\rho &= 2\sum_{i=1}^{m} g(\clb e_i, e_i) = m(m+1)\kappa. 
  \end{align}     
  Substituting the values of $\clb$ and $\overline\rho$ from \eqref{hscor1} and \eqref{hscor2} in \eqref{hsth1eq4} and \eqref{hsth1eq5}, we obtain the statistical Ricci tensor $\cl$ and the statistical scalar curvature $\rho$  of $M,$ respectively, as
   \begin{align}\label{hscor3} 
  g(\cl Y, Z) = \dfrac{(m-1)\kappa}{4} g(Y, Z)~~and~~ \rho = \dfrac{m(m-1)\kappa}{4}.
    \end{align}    
    Since $M$ is conformally flat by means of theorem \ref{hstheorem1}, hence, using \eqref{hsw2} and \eqref{hscor3}, we get the statistical curvature as
      \begin{align}\label{hscor4} 
     g(\cs (X, Y)Z, W) = \dfrac{\kappa}{4}[g(Y, Z)g(X, W) - g(X, Z) g(Y, W)],
     \end{align}        
     which proves that $\M$ is a space of constant curvature $\frac{\kappa}{4}$.
\end{proof}
%%%%%%%%%%%%%%%%%%%%%%%%%%%%%%%%%%%%%%%%%%%%%%%%%%%%%%%%%%%%%%%%%%
Now we prove our second theorem which states that if the holomorphic submanifold of a Bochner HSM is doubly autoparallel, then it has vanishing Bochner curvature tensor.

\begin{proof}[Proof of Theorem \ref{hstheorem2}]
Consider a local orthonormal frame field $\{e_1, ..., e_n, Je_1, ...,\\ Je_n, e_{n+1}, ..., e_m, Je_{n+1}, ..., Je_m\}$ of $T\M$ such that  $\{e_1, ..., e_n, Je_1, ..., Je_n\}$ and\\ $\{e_{n+1}, ..., e_m, Je_{n+1}, ..., Je_m\}$ are local orthonormal frame fields of $TM$ and $(TM)^\perp,$ respectively. Since $\M$ has vanishing Bochner curvature tensor, for $X, Y, Z \in \Gamma(T\M),$ using \eqref{hsge1}, \eqref{hsbf1}, we obtain the statistical sectional curvature of $M$ as
\begin{align}\label{bshs1}
\cs(X, Y)Z =& \frac{1}{2m+4}\{g(Y, Z)\clb X - g(\clb X, Z)Y + g(\clb Y, Z)X \nonumber \\ 
  &- g(X, Z)\clb Y + g(JY, Z)\clb JX - g(\clb JX, Z)JY + g(\clb JY, Z)JX \nonumber \\ 
  &-  g(JX, Z)\clb JY - 2 g(JX, \clb Y)JZ - 2 g(JX, Y)\clb JZ \} \nonumber \\ 
  &- \frac{\overline{\rho}}{(2m+4)(2m+2)}\{g(Y, Z)X - g(X, Z)Y + g(JY, Z)JX \nonumber \\ 
  &- g(JX, Z)JY - 2g(JX, Y)JZ\}.
\end{align}
Also, we know that the statistical Ricci curvature $\clb$ of $\M$ is given by
\begin{align}\label{bshs2}
g(\clb X, Y) =& g(\cl X, Y) - \sum_{\alpha=n+1}^{m}g(\csb(X, e_\alpha)Y, e_\alpha) - \sum_{\alpha=n+1}^{m}g(\csb(X, Je_\alpha)Y, Je_\alpha). 
\end{align}
Using \eqref{hsbf1} and the assumption that the Bochner curvature tensor of $\M$ vanishes, we evaluate the last two terms of \eqref{bshs2}. Substituting those values in \eqref{bshs2}, we obtain the statistical Ricci curvature $\clb$ as
\begin{align}\label{bshs3}
g(\clb X, Y) =& \dfrac{(m+2)}{(n+2)}g(\cl X, Y) + \dfrac{\epsilon}{(n+2)}g(X, Y),
\end{align}
where
\begin{align}\label{bshs4}
\epsilon = \frac{(n-m)}{2(m+1)}\overline{\rho} + \sum_{\alpha=n+1}^{m}g(\clb e_\alpha, e_\alpha).
\end{align}
%Taking trace in \eqref{bshs3}, we get the statistical scalar curvature $\overline\rho$ of $\M$ as
Using \eqref{bshs3}, we get
\begin{align}\label{bshs5111}
2\sum_{i=1}^{n}g(\clb e_i, e_i) = 2\sum_{i=1}^{n}\dfrac{(m+2)}{(n+2)}g(\clb e_i, e_i) + \dfrac{2n\epsilon}{(n+2)},
\end{align}
which gives the statistical scalar curvature $\overline\rho$ of $\M$ as
\begin{align}\label{bshs5}
\overline\rho = 2\sum_{\alpha=n+1}^{m}g(\clb e_\alpha, e_\alpha) + \dfrac{(m+2)}{(n+2)}\rho + \dfrac{2n\epsilon}{(n+2)}.
\end{align}
From \eqref{bshs4} and \eqref{bshs5}, we obtain
%Using \eqref{bshs4} in above equation, we obtain
\begin{align}\label{bshs6}
\dfrac{\overline\rho}{(m+1)(m+2)} = \dfrac{\rho}{(n+1)(n+2)} + \dfrac{4\epsilon}{(n+2)(m+2)}.
\end{align}
Using \eqref{bshs3} and \eqref{bshs6} in \eqref{bshs1}, we get the desired result.
\end{proof}

%%%%%%%%%%%%%%%%%%%%%%%%%%%%%%%%%%%%%%%%%%%%%%%%%%%%%%%%%%%%%%%%%% 
   \section{Conclusion}\label{hssec4}
   The results of this paper provide an initial understanding of the role of Bochner curvature in the statistical version of complex manifolds (HSM). The study of Lagrangian statistical submanifolds relates the vanishing of the Bochner curvature to their conformal behaviour, while the result for doubly autoparallel holomorphic submanifolds shows that this curvature property can be inherited by suitable submanifolds. Thus, Bochner curvature provides a useful geometric tool for investigating the interplay between the ambient holomorphic statistical structure and the geometry of its submanifolds. 
Since the contact analogue for Weyl tensor has been studied for certain contact manifolds, it would be interesting to extend this study to the contact statistical manifolds and their submanifolds in order to explore their conformal nature. 
      
            \section*{Acknowledgments}
The first author is thankful to UGC for providing financial assistance in terms of SRF scholarship vide NTA Ref. No.: $201610070797$(CSIR-UGC NET June 2020). The second author is thankful to the Department of Science and Technology (DST) Government of India for providing financial assistance in terms of FIST project (TPN-69301) vide the letter with Ref. No.: (SR/FST/MS-1/2021/104). 
\subsubsection*{Author contributions} All authors contributed equally.
\subsubsection*{Conflict of interest} There is no conflict of interest.
%\subsubsection*{Data Availability} There is no associate data.
%\textbf{Author contributions:} All authors contributed equally.\\	
%\textbf{Conflict of interest:} There is no conflict of interest.
%\subsection*{Author contributions} All authors contributed equally.
%
%\subsection*{Data Availability} There is no associate data.
%
%\subsection*{Declarations}
%
%\subsection*{Conflict of interest} There is no conflict of interest.

 \end{document}